\documentclass[12pt]{amsart}

\usepackage{cmap}
\usepackage{amsmath,amssymb,amsthm,amsfonts,mathtools,latexsym,proof}
\usepackage{graphicx,xcolor}
\usepackage{marginnote}
\usepackage{xspace,enumerate}
\usepackage{verbatim}
\usepackage[colorlinks,urlcolor=blue]{hyperref}

\theoremstyle{plain}

	\newtheorem{theorem}{Theorem}[subsection]
	\newtheorem{proposition}[theorem]{Proposition}
	\newtheorem{lemma}[theorem]{Lemma}
	\newtheorem{corollary}[theorem]{Corollary}

\theoremstyle{remark}

	\newtheorem{remark}{Remark}

\theoremstyle{definition}

	\newtheorem{definition}{Definition}[subsection]

	\newcommand{\rarr}{\ensuremath{\rightarrow}\xspace}
	\newcommand{\rsim}{\ensuremath{\mathord{\sim}}\xspace}
	\newcommand{\leql}{\ensuremath{\leqslant}\xspace}
	\newcommand{\Iff}{\quad if\textcompwordmark f\quad}
	\newcommand{\TBA}{\ensuremath{\mathrm{TBA}}\xspace}

	\newcommand{\Int}{\ensuremath{\mathsf{Int}}\xspace}
		\newcommand{\EInt}{\ensuremath{\mathcal{E}\Int}\xspace}
	\newcommand{\Cl}{\ensuremath{\mathsf{CL}}\xspace}

	\newcommand{\Sf}{\ensuremath{\mathsf{S4}}\xspace}
		\newcommand{\ESf}{\ensuremath{\mathcal{E}\Sf}\xspace}
	\newcommand{\Grz}{\ensuremath{\mathsf{Grz}}\xspace}

	\newcommand{\Nf}{\ensuremath{\mathsf{N4}^{\bot}}\xspace}
		\newcommand{\ENf}{\ensuremath{\mathcal{E}\Nf}\xspace}
	\newcommand{\Nk}{\ensuremath{\mathsf{NK}^{\bot}}\xspace}

	\newcommand{\Bsf}{\ensuremath{\mathsf{BS4}}\xspace}
		\newcommand{\EBsf}{\ensuremath{\mathcal{E}\Bsf}\xspace}

	\newcommand{\Tb}{\ensuremath{\mathrm{T}_{\mathbf{B}}}\xspace}
	\newcommand{\TGT}{\ensuremath{\mathrm{T}}\xspace}
	\newcommand{\taub}{\ensuremath{\tau_{\mathbf{B}}}\xspace}

	\newcommand{\langu}{\ensuremath{\mathcal{L}}\xspace}
	\newcommand{\langui}{\ensuremath{\langu_{i}}\xspace}
	\newcommand{\langus}{\ensuremath{\langu_{\rsim}}\xspace}
	\newcommand{\langub}{\ensuremath{\langu^{\Box}}\xspace}
	\newcommand{\langubs}{\ensuremath{\langu_{\rsim}^{\Box}}\xspace}

	\newcommand{\Fml}{\ensuremath{\mathrm{Fm}_{\langu}}\xspace}
	\newcommand{\Fmli}{\ensuremath{\mathrm{Fm}_{\langui}}\xspace}
	\newcommand{\Fmls}{\ensuremath{\mathrm{Fm}_{\langus}}\xspace}
	\newcommand{\Fmlb}{\ensuremath{\mathrm{Fm}_{\langub}}\xspace}
	\newcommand{\Fmlbs}{\ensuremath{\mathrm{Fm}_{\langubs}}\xspace}

	\newcommand{\bfA}{\ensuremath{\mathbf{A}}\xspace}
	\newcommand{\FdA}{\ensuremath{\mathcal{F}_{d}(\bfA)}\xspace}
	\newcommand{\IA}{\ensuremath{\mathcal{I}(\bfA)}\xspace}
	\newcommand{\clA}{\ensuremath{\mathcal{A}}\xspace}

	\newcommand{\bfB}{\ensuremath{\mathbf{B}}\xspace}
	\newcommand{\FbB}{\ensuremath{\mathcal{F}_{\Box}(\bfB)}\xspace}
	\newcommand{\IdiB}{\ensuremath{\mathcal{I}_{\Diamond}(\bfB)}\xspace}
	\newcommand{\clB}{\ensuremath{\mathcal{B}}\xspace}

	\newcommand{\GsB}{\ensuremath{\mathsf{G}(\mathbf{B})}\xspace}
	\newcommand{\GB}{\ensuremath{\mathcal{G}(\mathbf{B})}\xspace}

	\newcommand{\TL}{\ensuremath{\mathcal{T}(L)}\xspace}
	\newcommand{\TM}{\ensuremath{\mathcal{T}(M)}\xspace}
	\newcommand{\WL}{\ensuremath{\mathcal{W}_{L}}\xspace}
		\newcommand{\nablaL}{\ensuremath{\nabla_{L}}\xspace}
		\newcommand{\DeltaL}{\ensuremath{\Delta_{L}}\xspace}
	\newcommand{\WM}{\ensuremath{\mathcal{W}_{M}}\xspace}
		
		\newcommand{\DeltaM}{\ensuremath{\Delta_{M}}\xspace}

	\newcommand{\TwBND}{\ensuremath{Tw(\bfB, \nabla, \Delta)}\xspace}
	\newcommand{\bfT}{\ensuremath{\mathbf{T}}\xspace}
	\newcommand{\GsT}{\ensuremath{\mathsf{G}_{2}(\bfT)}\xspace}
	\newcommand{\GmT}{\ensuremath{\Gamma(\bfT)}\xspace}
	\newcommand{\Lnabla}{\ensuremath{\Lambda(\bfB, \nabla)}\xspace}
	\newcommand{\nabg}{\ensuremath{\nabla_{\mathsf{G}}(\bfT)}\xspace}
	\newcommand{\Delg}{\ensuremath{\Delta_{\mathsf{G}}(\bfT)}\xspace}
	\newcommand{\GT}{\ensuremath{\mathcal{G}_{2}(\bfT)}\xspace}

\title[On Modal Companions of Logics with Strong Negation]{On Modal Companions of Logics with Strong Negation}

\author{Dmitry~M.~Anishchenko}

	\address{Dmitry~M.~Anishchenko,
		Steklov Mathematical Institute of Russian Academy of Sciences,
		Moscow, Russia.}
		
	\email{d.anishchenko@g.nsu.ru}

\begin{document}

\begin{abstract}

	\Bsf is a natural Belnapian conservative extension of Lewis' modal system \Sf via strong negation. In ~\cite{OdWans2010} it was proved that the translation \Tb that naturally generalises the G\"{o}del--Tarski translation \TGT  embeds faithfully Nelson's logic \Nf into \Bsf. So it is natural to define a modal companion of a logic extending \Nf as an extension of \Bsf. In this paper we construct a representation of an \Nf-lattice similar to the representation of a Heyting algebra as an open elements algebra for a suitable topoboolean algebra. Using this algebraic result we construct a wide class of \Nf-extensions, elements of which have modal companions. In particular, all $\mathsf{N3}^\bot$-extensions have modal companions. Also we prove that there are a continuum of \Nf-extensions that have no modal companions.

\end{abstract}

\maketitle

\textbf{Keywords:} algebraic logic, strong negation, Belnapian modal logic, modal companion, twist-structure, Kleene algebra.

\section{Introduction}

Since the discovery that intuitionistic logic \Int is faithfully embedded into the normal modal logic \Sf via G\"{o}del--Tarski translation (see \cite[Theorem 3.83]{ChagrovZ} for details), the notion of modal companion is linked with extensions of \Int and normal extensions of \Sf.  Transferring different properties from a superintuitionistic logic to its modal companions and {\em vice versa} became an important branch of investigations. The picture of results obtained in this branch is well-presented in the survey \cite{ChagZach92}.

Our paper will be devoted to modal companions for extensions of Nelson's constructive logic. Originally D.~Nelson \cite{Nelson49} defined a constructive logic with stong negation, which is denoted now by $\mathsf{N3}^\bot$, as an alternative formalization of Brouwer's intuitionistic logic. Later R.~Thomason \cite{Thomason} developed a Kripke style semantics for $\mathsf{N3}^\bot$, N.~Belnap suggested his ``useful four-valued logic'' \cite{Belnap77} based on what is called now the Belnap--Dunn matrix {\sffamily BD4}, and M.~Dunn \cite{Dunn76} proved that {\sffamily BD4} provides a semantics for First-Degree Entailment {\sffamily FDE}. These facts allows to consider $\mathsf{N3}^\bot$ as well as its paraconsistent version \Nf as many-valued versions of intuitionistic logic. More exactly, Nelson's constructive logic \Nf is determined by the same class of Kripke frames as intuitionistic logic, but models of this logic are augmented with four-valued valuations. In each possible world a formula may have one of four truth values:
$\mathsf{T}$ (True), $\mathsf{F}$ (False), $\mathsf{N}$ (Neither), 
$\mathsf{B}$ (Both) of the Belnap--Dunn matrix {\sffamily BD4}; and values of conjunction, disjunction and strong negation of formulas are computed via the respective operations of {\sffamily BD4}. In case of $\mathsf{N3}^\bot$ the truth value $\mathsf{B}$ is excluded, a formula can not be true and false simultaneously in each world. 
		
The Belnapian modal logic {\sffamily BK} \cite{OdWans2010} extends the least normal modal logic {\sffamily K} via strong negation \rsim in exactly the same way as the logic \Nf extends \Int. The Kripke style semantics for {\sffamily BK} also can be obtained from that of {\sffamily K} via replacement of two-valued valuations by four-valued ones. The semantics of necessity and possibility operators is defined according to Fitting's approach \cite{Fitting1991} to many-valued modalities. An algebraic semantics for {\sffamily BK} and its extensions was developed in \cite{OdLat2012} and \cite{Speranski13}. The  structure of the lattice of logics extending {\sffamily BK} was investigated in \cite{OdSp2016}. Further, first-order versions of {\sffamily BK} were studied in \cite{GrefSp24}, and the "bilattice" version of {\sffamily BK} was considered in \cite{Speranski21}. It should be noted that only two of four truth values of the Belnap-Dunn matrix can be expressed in the language of {\sffamily BK}: $\mathsf{T}$ and $\mathsf{F}$; but enriching the language with the constants $\mathsf{N}$ and $\mathsf{B}$ leads to unexpected results (see \cite{OdSp2020} for details).

It was proved in \cite{OdWans2010} that \Nf is faithfully embedded into the logic
$\Bsf = \mathsf{BK} + \{ \Box p \rarr p, \Box p \rarr \Box \Box p \}$ via the translation \Tb, which is a natural extension of the well-known
G\"{o}del--Tarski translation \TGT \cite{Goedel1933} to the language with strong negation. Due to this result it would be natural to define a modal companion of an \Nf-extension as a \Bsf-extension and we are curious about transferring results about modal companions of superintuitionistic logics to extensions of \Nf. The efforts to study modal companions of \Nf-extensions were made in \cite{Kaushan2019} and \cite{OdinVish24}. These efforts demonstrated that finding modal companions of constructive logics meets essential complications. As a result in \cite{OdinVish24} it only was proved that the so-called special extensions of \Nf have modal companions.
		
In this paper we prove that logics from a wide class of \Nf-extensions, which includes, in particular, all $\mathsf{N3}^\bot$-extensions, have modal companions. Also we find a continuum of logics that  have no modal companions. We get the results by constructing some regular presentation of models of \Nf as subsets of models of \Bsf, which is similar to the presentation of Heyting algebras as open elements algebras of topoboolean algebras. 
		
The paper is structured as follows. In Section~2 we define the logics of interest, their algebraic semantics and the notion of modal companion. In Section~3 we investigate algebraic relations between models of \Bsf and models of \Nf. In Section~4 we distinguish a wide class of logics and prove that all logics from this class have modal companions by using algebraic results from Section~3; also we show that there are a continuum of logics without modal companions.

\section{Preliminaries}

In the present section we give definitions of logics of interest and of their algebraic semantics. We recall the classical results about modal companions of superintuitionistic logics and provide definitions of the translation \Tb used to define modal companions of \Nf-extensions.
We will assume throughout that the reader is familiar with basic notions of universal algebra. We refer to \cite{BurrSank} for details.

\subsection{Logics}

We work with the following propositional languages:
\begin{align*}
	\langui &= \{ \wedge^{2}, \vee^{2}, \rarr^{2}, \bot^{0} \},
	&
	\langus &= \langui \cup \{ \rsim^{1} \},
	\\
	\langub &= \langui \cup \{ \Box^{1} \},
	&
	\langubs &= \langui \cup \{ \rsim^{1}, \Box^{1}, \Diamond^{1} \}.
\end{align*}

We denote a countable set of propositional variables by
$\mathrm{Prop}$. Let $\Fml$ be the absolutely free algebra {(the formula-algebra)} of a language \langu built up over
$\mathrm{Prop}$. {\em Formulas} of the language \langu are elements of
$\Fml$.

A set $L \subseteq \Fml$ is called a
{\em logic in the language} \langu if it satisfies the following property:
\begin{enumerate}
	
	\item[1.] For $\langu \in \{ \langui, \langus \}$,
	$L$ is closed under the \textit{substitution} $(\mathrm{SUB})$ and the
	\textit{modus ponens} $(\mathrm{MP})$ rules:
	\[
		\mbox{(SUB)} \quad 
		\dfrac{ \varphi(p_1, \dots, p_n) }											  	  { \varphi(\psi_1, \dots, \psi_n) }
		\ \ \ \
		\mbox{(MP)} \quad
		\dfrac{ \varphi, \quad \varphi \rarr \psi }
			  {\psi}
	\]
	
	\item[2.] In case $\langu = \langub$,
	$L$ is closed under $(\mathrm{SUB})$, $(\mathrm{MP})$ and
	the \textit{monotonicity rule for $\Box$} $(\mathrm{RM}_{\Box})$:
	\[
		(\mbox{RM}_{\Box}) \quad 													\dfrac{ \varphi \rarr \psi }
			  { \Box \varphi \rarr \Box \psi }
	\]
	
	\item[3.] Finally, if $\langu = \langubs$, then
	$L$ is closed under $(\mathrm{SUB})$, $(\mathrm{MP})$,
	$(\mathrm{RM}_{\Box})$ and the \textit{monotonicity rule for
	$\Diamond$} $(\mathrm{RM}_{\Diamond})$:
	\[
		(\mbox{RM}_{\Diamond}) \quad 												\dfrac{ \varphi \rarr \psi }
			  { \Diamond \varphi \rarr \Diamond \psi }
	\]
	
\end{enumerate}

For a logic $L$ in a language \langu we will denote the set of all logics in \langu containing $L$ by $\mathcal{E}L$. $\mathcal{E}L$ is a complete lattice w.r.t. inclusion, because the intersection of an arbitrary set of logics in \langu is also a logic in \langu. So $\mathcal{E}L$ is called the {\em lattice of extensions of $L$}. Also for a set of formulas
$X \subseteq \Fml$ we will denote the least logic in \langu containing $L$ and $X$ by $L+X$.

We use the following standard abbreviations:
\begin{align*}
	\neg \varphi & := \varphi \rarr \bot
	&&\text{(intuitionistic negation);} \\
	\Diamond \varphi & := \neg \Box \neg \varphi
	&&\text{(only for formulas in \Fmlb);} \\
	\varphi \leftrightarrow \psi & :=
		(\varphi \rarr \psi) \wedge (\psi \rarr \varphi)
	&&\text{(equivalence);} \\
	\varphi \Leftrightarrow \psi & :=
		(\varphi \leftrightarrow \psi)
		\wedge
		(\rsim \varphi  \leftrightarrow \rsim \psi)
	&&\text{(strong equivalence).}
\end{align*}

Now we introduce logics of interest. The {\em intuitionistic logic} \Int is the least logic in the language \langui containing the following axioms:
\begin{enumerate}
	
	 \item[1.] $ p \rarr (q \rarr p) $

	 \item[2.] $ (p \rarr (q \rarr r)) \rarr 
	 			 ((p \rarr q) \rarr (p \rarr r)) $

	 \item[3.] $ (p \wedge q) \rarr p $

	 \item[4.] $ (p \wedge q) \rarr q $

	 \item[5.] $ p \rarr (q \rarr (p \wedge q)) $

	 \item[6.] $ p \rarr (p \vee q) $

	 \item[7.] $ q \rarr (p \vee q) $

	 \item[8.] $ (p \rarr r) \rarr 
	 		     ( (q \rarr r) \rarr 
	 					((p \vee q) \rarr r) ) $

	 \item[9.] $ \bot \rarr p $
	
\end{enumerate}
The {\em classical logic} is defined as $\Cl = \Int + \{ p \vee \neg p \}$.

{\em Constructive Nelson's logic} \Nf is the least logic in the language
\langus containing:
\begin{enumerate}[I)]
	
	\item the axioms of \Int
	
	\item the axioms of strong negation:
	\begin{enumerate}[1.]
		
		\item $ \rsim (p \vee q) \leftrightarrow
				(\rsim p \wedge \rsim q) $

		\item $ \rsim (p \wedge q) \leftrightarrow
				(\rsim p \vee \rsim q) $

		\item $ \rsim (p \rarr q) \leftrightarrow
				(p \wedge \rsim q) $

		\item $ \rsim \rsim p \leftrightarrow p $

		\item $ \rsim \bot $
		
	\end{enumerate}
	
\end{enumerate}

The modal logic \Sf is the least logic in the language \langub containing:
\begin{enumerate}[I)]
	
	\item the axioms of \Cl
	
	\item the following modal axioms:
	\begin{enumerate}[1.]
		
		\item $ \Box (p \rarr p) $

		\item $ (\Box p \wedge \Box q) \rarr
				 \Box (p \wedge q) $

		\item $ \Box p \rarr p $

		\item $ \Box p \rarr \Box \Box p $
		
	\end{enumerate}
	
\end{enumerate}
{\em Grzegorczyk's logic} is defined as
$ \Grz = \Sf + \{ \Box ( \Box (p \rarr \Box p) \rarr p) \rarr p \} $.

The modal logic \Bsf is the least logic in the language \langubs containing:
\begin{enumerate}[I)]
	
	\item the axioms of \Cl
	
	\item the axioms of strong negation
	
	\item the modal axioms of \Sf
	
	\item the interplay axioms:
	\begin{enumerate}[1.]
		
		\item $ \neg \Box p \leftrightarrow \Diamond \neg p $
		
		\item $ \neg \Diamond p \leftrightarrow \Box \neg p $
		
		\item $ \Box p \Leftrightarrow \rsim \Diamond \rsim p $
		
		\item $ \Diamond p \Leftrightarrow \rsim \Box \rsim p $
		
	\end{enumerate}
	
\end{enumerate}

It is known (see \cite{Odintsov05a}) that each $L \in \ENf$ is closed under the positive replacement rule
\[
	(\mbox{PR}) \quad
	\frac{
			\varphi_{1} \leftrightarrow \psi_{1},
			\dots,
			\varphi_{n} \leftrightarrow \psi_{n}
		 }
		 {
			\chi(\varphi_{1}, \dots, \varphi_{n})
			\leftrightarrow
			\chi(\psi_{1}, \dots, \psi_{n})
		 },
\]
where $\varphi_{j}, \psi_{j} \in \Fmls$ and $\chi \in \Fmli$. The logic \Nf is not closed under the usual replacement rule without restriction on $\chi$, but each
$L \in \ENf$ is closed under the weak replacement rule
\[
	(\mbox{WR}) \quad
	\frac{
			\varphi_{1} \Leftrightarrow \psi_{1},
			\dots,
			\varphi_{n} \Leftrightarrow \psi_{n}
		 }
		 {
			\chi(\varphi_{1}, \dots, \varphi_{n})
			\leftrightarrow
			\chi(\psi_{1}, \dots, \psi_{n})
		 },
\]
where $\varphi_{j}, \psi_{j}, \chi \in \Fmls$.

For modal logics it is easy to check that each $L \in \ESf$ or
$L \in \EBsf$ is closed under the normalisation rule
\[ (\mbox{RN}) \quad \frac{\varphi}{\Box \varphi}. \] The logic \Bsf is not closed under the usual replacement rule, similarly to \Nf, but (see \cite{OdWans2010}) each logic $L \in \EBsf$ is closed under the positive and the weak variants of this rule:
\[
	(\mbox{PR}) \quad
	\frac{
			\varphi_{1} \leftrightarrow \psi_{1},
			\dots,
			\varphi_{n} \leftrightarrow \psi_{n}
		}
		{
			\theta(\varphi_{1}, \dots, \varphi_{n})
			\leftrightarrow
			\theta(\psi_{1}, \dots, \psi_{n})
		},
\] and
\[
	(\mbox{WR}) \quad
	\frac{
			\varphi_{1} \Leftrightarrow \psi_{1},
			\dots,
			\varphi_{n} \Leftrightarrow \psi_{n}
		}
		{
			\chi(\varphi_{1}, \dots, \varphi_{n})
			\leftrightarrow
			\chi(\psi_{1}, \dots, \psi_{n})
		},
\]
where $\theta \in \Fmlb$ and $\varphi_{j}, \psi_{j}, \chi \in \Fmlbs$.

 \subsection{Algebraic semantics}

Note that we use the same abbreviations for operations $\neg, \Diamond, \leftrightarrow, \Leftrightarrow$ in algebras $\bfA = \langle A; \langu^{\bfA} \rangle $ as for formulas.

An algebra
$\bfA = \langle A; \wedge_{\bfA}, \vee_{\bfA},
\rarr_{\bfA}, \bot_{\bfA} \rangle $ is called a {\em Heyting algebra}, if
its reduct $\langle A; \wedge_{\bfA}, \vee_{\bfA} \rangle$ is a lattice
with order relation $\leql_{\bfA}$, $\bot_{\bfA}$ is the least element w.r.t. $\leql_{\bfA}$, and for arbitrary $a, b, c \in A$ the following holds:
\[
	a \wedge_{\bfA} b \leql_{\bfA} c
	\iff
	a \leql_{\bfA} b \rarr_{\bfA} c
\]
Sometimes we will omit the subscript \bfA, when it is clear which algebra we work in.

Note that each Heyting algebra \bfA is a distributive lattice, it has the greatest element $1_{\bfA}$, and for all $a \in A$ we have
$a \rarr a = 1_{\bfA}$ (see, e.g., \cite{RasSikor72}). Also the following hold for all
$a, b \in A$:
\begin{enumerate}[(1)]
	
	\item $a \rarr b = 1_{\bfA} \iff a \leql_{\bfA} b$
	
	\item $a \wedge (a \rarr b) = a \wedge b$
	
	\item $a \leql_{\bfA} b
			\Longrightarrow
			\neg b \leql_{\bfA} \neg a$
	
	\item $a \leql_{\bfA} \neg \neg a$
	
\end{enumerate}

A non-empty set $F \subseteq A$ is called a {\em filter of} \bfA if it satisfies the following properties:
\[
	(1)\ b \in F,\ b \leql a \Rightarrow a \in F;
	\qquad
	(2)\ a, b \in F \Rightarrow a \wedge b \in F.
\]
By analogy, a non-empty set $I \subseteq A$ is called an {\em ideal of} \bfA if the dual properties hold:
\[
	(1)\ b \in I,\ a \leql b \Rightarrow a \in I;
	\qquad
	(2)\ a, b \in I \Rightarrow a \vee b \in I.
\]
The set of all filters (ideals) of \bfA is denoted by
$\mathcal{F}(\mathbf{A})$ ($\IA$).

An element $a \in A$ is called {\em dense} if one of the following three equivalent properties holds:
\[
(1)\ \neg a = \bot ;\quad
(2)\ \neg \neg a = 1 ;\quad
(3)\ a = b \vee \neg b \text{\,\, for some $b \in A$}.
\]
It is well-known that the set of all dense elements of \bfA is a filter, and we denote it by $F_{d}(\bfA)$. We denote the set of all filters which contain $F_{d}(\bfA)$ as a subset by \FdA.

A Heyting algebra \bfA is called a {\em Boolean algebra} if
$ a \vee \neg a = 1 $ for all $a \in A$ or, equivalently, if
$F_{d}(\bfA) = \{ 1 \}$.

An algebra
$\bfB = \langle B; \wedge_{\bfB}, \vee_{\bfB},
\rarr_{\bfB}, \bot_{\bfB}, \Box_{\bfB} \rangle $ is called a {\em topological Boolean algebra} or a \TBA if its reduct
$\langle B; \wedge_{\bfB}, \vee_{\bfB}, \rarr_{\bfB}, \bot_{\bfB} \rangle$ is a Boolean algebra and the following hold for all $a, b \in B$ (some subscripts are omitted):
\begin{gather*}
		\Box 1_{\bfB} = 1_{\bfB},
	\qquad
		\Box (a \wedge b) = \Box a \wedge \Box b,
	\\
		\Box a \leql_{\bfB} a,
	\qquad
		\Box a \leql_{\bfB} \Box \Box a.
\end{gather*}
Note that in an arbitrary \TBA \bfB the following hold for all $a, b \in B$
(see, e.g., \cite{RasSikor72}):
\begin{enumerate}[(1)]
	
	\item $\Diamond \bot_{\bfB} = \bot_{\bfB}, \
			\Box \bot_{\bfB} = \bot_{\bfB},\
			\Diamond 1_{\bfB} = 1_{\bfB}$
	
	\item $\Diamond (a \vee b) = \Diamond a \vee \Diamond b$
	
	\item $a \leql_{\bfB} \Diamond a$
	
	\item $\Diamond \Diamond a \leql_{\bfB} \Diamond a$
	
	\item $a \leql_{\bfB} b
			\Longrightarrow(
				\Box a \leql_{\bfB} \Box b
				\text{ and }
				\Diamond a \leql_{\bfB} \Diamond b
			)$
	
	\item $\Box \Box a = \Box a
			\text{ and }
			\Diamond \Diamond a = \Diamond a$
	
	\item $\Box a \vee \Box b \leql_{\bfB} \Box ( a \vee b )
			\text{ and }
			\Diamond (a \wedge b) \leql_{\bfB}
				\Diamond a \wedge \Diamond b$
	
\end{enumerate}
A {\em filter (ideal) of} \TBA is defined as for Heyting algebras. We call a filter (ideal) of \bfB\ {\em open (closed)} if it is closed under
$\Box_{\bfB}$ ($\Diamond_{\bfB}$). The set of all open filters (closed ideals) of \bfB is denoted by
$\FbB$ ($\IdiB$).

Let \bfB be a \TBA. The set $\GsB := \{ a \in B \mid \Box a  = a \}$ of all
{\em open elements} contains $\bot_{\bfB}$ and is closed under
$\vee, \wedge$. The algebra
\[
	\GB = \langle \GsB;
	\vee_{\GB}, \wedge_{\GB}, \rarr_{\GB}, \bot_{\GB} \rangle
\]
where
\begin{align*}
	a \vee_{\GB} b & := a \vee b
	&
	a \wedge_{\GB} b & := a \wedge b \\
	a \rarr_{\GB} b & := \Box (a \rarr b)
	&
	\bot_{\GB} & := \bot_{\bfB}
\end{align*}
is called the {\em algebra of open elements of} \bfB.
It is known that \GB is a Heyting algebra (see, e.g., \cite[pp.~61--62]{GabMaks05}).

For a \TBA \bfB we denote its subalgebra generated by all open elements by
$\bfB^{s}$. It is well-known that for each Heyting algebra
\bfA there is one up to isomorphism \TBA $\mathbf{C}$ such that
$\mathcal{G}(\mathbf{C}) = \bfA$ and
$\mathbf{C}^{s} = \mathbf{C}$ (see, e.g., \cite[Theorem\,3.8]{GabMaks05}). We denote this \TBA by $s(\bfA)$.

Let $\bfA = \langle A, \vee, \wedge, \rarr, \bot \rangle $ be a Heyting algebra. The algebra
$\bfA^{\bowtie} = \langle A \times A;
\vee_{\bfA^{\bowtie}}, \wedge_{\bfA^{\bowtie}}, \rarr_{\bfA^{\bowtie}}, \bot_{\bfA^{\bowtie}}, \rsim_{\bfA^{\bowtie}} \rangle$ is called the
{\em full twist-structure over a Heyting algebra} \bfA, where the operations of $\bfA^{\bowtie}$ are defined as follows:
\begin{align*}
	(a,b) \vee_{\bfA^{\bowtie}} (c,d) &= (a \vee c, b \wedge d) &
	\rsim_{\bfA^{\bowtie}} (a,b) & = (b,a) \\	
	(a,b) \wedge_{\bfA^{\bowtie}} (c,d) &= (a \wedge c, b \vee d) &
	\bot_{\bfA^{\bowtie}} & = (\bot, 1) \\
	(a,b) \rarr_{\bfA^{\bowtie}} (c,d) &=  ( a \rarr c, a \wedge d)
\end{align*}
If we have a \TBA $\bfB = \langle B, \vee, \wedge,
\rarr, \bot, \Box \rangle $, then the algebra
$\bfB^{\bowtie} = \langle B \times B;
\vee_{\bfB^{\bowtie}}, \wedge_{\bfB^{\bowtie}}, \rarr_{\bfB^{\bowtie}}, \bot_{\bfB^{\bowtie}}, \rsim_{\bfB^{\bowtie}},
\Box_{\bfB^{\bowtie}}, \Diamond_{\bfB^{\bowtie}} \rangle$ is called
the {\em full twist-structure over a TBA} \bfB, where the operations
of \langus are defined as for the full twist-structure over a Heyting algebra, and the modal operations are defined as:
\begin{align*}
	\Box_{\bfB^{\bowtie}} (a, b) & =  (\Box a,  \Diamond  b)
	\quad & \Diamond_{\bfB^{\bowtie}}  (a, b) & =  ( \Diamond  a, \Box b)
\end{align*}
Finally, if we have a Heyting algebra (a \TBA) $\mathbf{C}$, then a subalgebra $\clB$ of $\mathbf{C}^{\bowtie}$ is called a
{\em twist-structure over} $\mathbf{C}$ if $\pi_{1}(\clB) = C$, where
$\pi_{1}$ denotes the first projection function. Also we denote the second projection function by $\pi_{2}$ in what follows.

We define the following sets for a twist-structure \clB over a Heyting algebra (a \TBA) $\mathbf{C}$:
\[
	\nabla(\clB) := \{ a \vee b \mid (a, b) \in \clB \},
	\quad
	\Delta(\clB) := \{ a \wedge b \mid (a, b) \in \clB \}.
\]
It is known (see \cite{ConstrNeg, OdLat2012}) that
$\nabla(\clB) \in \mathcal{F}_{d}(\mathbf{C})$
($\mathcal{F}_{\Box}(\mathbf{C})$) and
$\Delta(\clB) \in \mathcal{I}(\mathbf{C})$
($\mathcal{I}_{\Diamond}(\mathbf{C})$).
Conversely, for every
$\nabla \in \mathcal{F}_{d}(\mathbf{C})$
($\mathcal{F}_{\Box}(\mathbf{C})$) and
$\Delta \in \mathcal{I}(\mathbf{C})$
($\mathcal{I}_{\Diamond}(\mathbf{C})$) we can define the set
\[
	X = \{ (a, b) \in C \times C \mid
		a \vee b \in \nabla, a \wedge b \in \Delta \},
\]
and it is known that $X$ is closed under the operations of
$\mathbf{C}^{\bowtie}$ and $\pi_{1}(X) = C$. So we can define the twist-structure over $\mathbf{C}$ with $X$ as the domain, and we denote it by $Tw(\mathbf{C}, \nabla, \Delta)$. Moreover, it is known that
\[
	\nabla(Tw(\mathbf{C}, \nabla, \Delta)) = \nabla,
	\quad
	\Delta(Tw(\mathbf{C}, \nabla, \Delta)) = \Delta
\]
and
\[
	Tw(\mathbf{C}, \nabla(\clB), \Delta(\clB)) = \clB,
\]
so twist-structures are uniquely determined by these sets, which we call the
{\em invariants of twist-structures}. The sets $\nabla(\clB)$ and $\Delta(\clB)$ are called the {\em filter of \clB} and the {\em ideal of \clB} respectively.

Now, let
$\langu \in
	\{ \langui, \langub, \langus, \langubs \}$,
$\Gamma \cup \{ \varphi \} \subseteq \Fml$, and let
$\clA\ (\mathcal{K})$ be:
\begin{enumerate}[1.]
	
	\item a Heyting algebra (some class of Heyting algebras), if $\langu = \langui$;
	
	\item a \TBA (some class of {\TBA}s), if $\langu = \langub$;
	
	\item a twist-structure over a Heyting algebra \bfA (some class of twist-struc\-tures over Heyting algebras), if $\langu = \langus$;
	
	\item a twist-structure over a \TBA \bfA (some class of twist-structures over {\TBA}s), if $\langu = \langubs$.
	
\end{enumerate}
Also let us denote the class of all Heyting algebras (the class of all twist-structures over Heyting algebras) by $\mathcal{HA}$ ($\mathcal{HA}^{\bowtie}$) and the class of all {\TBA}s (the class of all twist-structures over {\TBA}s) by $\mathcal{TBA}$ ($\mathcal{TBA}^{\bowtie}$). Recall that for an algebra
$\clA = \langle A; \langu^{\clA} \rangle$ an {\em $\clA$-valuation} is a homomorphism
$v \colon \Fml \rarr \clA$.

Let's define the validity of formulas (sets of formulas) in algebras (classes of algebras):

\begin{definition}\hfill
	
	\begin{enumerate}[1.]
		
		\item If $\langu \in \{ \langui, \langub \}$,\ \
		$
			\clA \models \varphi \text{\Iff}
			v(\varphi) = 1_{\clA} \text{ for every \clA-valuation $v$}
		$
		
		\item If $\langu \in \{ \langus, \langubs \}$,
		\ \
		$
			\clA \models \varphi \text{\Iff}
			\pi_{1}v(\varphi) = 1_{\bfA} \text{ for every
			$\mathcal{A}$-valuation $v$}
		$
		
		\item $\clA \models \Gamma \text{\Iff} 
		\clA \models \varphi \text{ for every $\varphi \in \Gamma$}$
		
		\item $\mathrm{L} \clA :=
		\{ \varphi \in \Fml \mid \clA \models \varphi \}$	
		
		\item $\mathcal{K} \models \varphi \text{\Iff}
		\clA \models \varphi \text{ for every $\clA \in \mathcal{K}$}$
		
		\item $\mathcal{K} \models \Gamma \text{\Iff}
		\mathcal{K} \models \varphi \text{ for every $\varphi \in \Gamma$}$
		
		\item $\mathrm{L} \mathcal{K} := \bigcap
			\{ \mathrm{L} \clA \mid \clA \in \mathcal{K} \} $
		
	\end{enumerate}
	
\end{definition}

\begin{remark}\label{rem1}
	
	Note that if  we have some twist-structure \clA over a Heyting algebra (a \TBA) and some \clA-valuation $v$, then
$\pi_{1}(v(\psi(p_{1}, \dots, p_{n}))) =
\psi(\pi_{1}(v(p_{1})), \dots, \pi_{1}(v(p_{n})))$ for every
$\psi \in \Fmli$ ($\psi \in \Fmlb$), because the operations of twist-structure are agreed with the operations of Heyting algebra (\TBA) on the first component.
	
\end{remark}

The following theorem links logics and algebras:
\begin{theorem}\hfill
	
	\begin{enumerate}[1.]
		
		\item $\mathrm{L} \mathcal{HA} = \Int$ and
		$\mathrm{L} \clA \; (\mathrm{L} \mathcal{K})
		\in \EInt$, if $\langu = \langui$ (see, e.g.,
		\cite{RasSikor72}).
		
		\item $\mathrm{L} \mathcal{TBA} = \Sf$ and
		$\mathrm{L} \clA \; (\mathrm{L} \mathcal{K})
		\in \ESf$, if $\langu = \langub$ (see, e.g.,
		\cite{ChagrovZ}).
		
		\item $\mathrm{L} \mathcal{HA}^{\bowtie} = \Nf$ and
		$\mathrm{L} \clA \; (\mathrm{L} \mathcal{K})
		\in \ENf$, if $\langu = \langus$ (see \cite{ConstrNeg}).
		
		\item $\mathrm{L} \mathcal{TBA}^{\bowtie} = \Bsf$ and
		$\mathrm{L} \clA \; (\mathrm{L} \mathcal{K})
		\in \EBsf$, if $\langu = \langubs$ (see \cite{OdWans2010}).
		
	\end{enumerate}
	
\end{theorem}

We call an algebra \clA a {\em model of logic
$L$ in the language \langu} if $\clA \models L$. A model \clA of $L$ is called {\em proper} if $\mathrm{L} \clA = L$.

For each language $\langu \in \{ \langus, \langubs \}$ and logic
$L \in \ENf$ or $L \in \EBsf$ in the language \langu we define a special twist-structure which is called the {\em Lindenbaum twist-structure of} $L$.
First, we define a binary relation $\equiv_{L}$ on \Fml: 
\[
	\varphi \equiv_{L} \psi \iff \varphi \leftrightarrow \psi \in L.
\]
It is obvious that $\equiv_{L}$ is an equivalence relation. Denote the equivalence class of $\varphi$ w.r.t. $\equiv_{L}$ by $[\varphi]_{L}$. We define the following operations $\vee_{\TL}$, $\wedge_{\TL}$, $\rarr_{\TL}$ and $\bot_{\TL}$
(also $\Box_{\TL}$ and $\Diamond_{\TL}$) on the equivalence classes:
\begin{align*}
	[\varphi]_{L} \vee_{\TL} [\psi]_{L} & := [\varphi \vee \psi]_{L} 
	&
	[\varphi]_{L} \wedge_{\TL} [\psi]_{L} & := [\varphi \wedge \psi]_{L}
	\\
	[\varphi]_{L} \rarr_{\TL} [\psi]_{L} & := [\varphi \rarr \psi]_{L}
	&
	\bot_{\TL} & := [\bot]_{L}
	\\
	(\Box_{\TL} [\varphi]_{L} & := [\Box \varphi]_{L})
	&
	(\Diamond_{\TL} [\varphi]{L} & := [\Diamond \varphi]_{L})
\end{align*}
These operations are well-defined because $L$ is closed under \mbox{(PR)}. Also the definition of $\Diamond_{\TL}$ is correct w.r.t.
$\Diamond \varphi := \neg \Box \neg \varphi$ because of the interplay axioms 1--2.

$\Int \subseteq \Nf$ ($\Sf \subseteq \Bsf$) implies that the algebra
\[
	\TL := \left\langle
		\{ [\varphi]_{L} \mid \varphi \in \Fml \}; \vee_{\TL}, \wedge_{\TL},
		\rarr_{\TL}, \bot_{\TL} (, \Box_{\TL})
	\right\rangle,
\] 
is a Heyting algebra (a \TBA).
And in that case
\[ 
	1_{\TL} = [\varphi]_{L} \rarr [\varphi]_{L} =
	[\varphi \rarr \varphi]_{L} = L\,.
\]

It is known that the following subset of 
$\TL^{2}$ is closed under the operations of $\TL^{\bowtie}$ (see \cite{ConstrNeg, OdWans2010}):
\[
	W_{L} := \{ ([\varphi]_{L}, [\rsim \varphi]_{L}) \mid
			\varphi \in \Fml \}.
\]
It is obvious that
$\pi_{1}(W_{L})
= \{ [\varphi]_{L} \mid \varphi \in \Fml \}$,
so the {\em Lindenbaum twist-structure} \WL is the twist-structure over \TL with $W_{L}$ as the domain. We denote its invariants by
\begin{align*}
	\nablaL & := \nabla(\WL) =
		\{ [\varphi \vee \rsim \varphi]_{L} \mid 
			\varphi \in \Fml \},
	\\
	\DeltaL & := \Delta(\WL) =
		\{ [\varphi \wedge \rsim \varphi]_{L} \mid 
			\varphi \in \Fml \}.
\end{align*}
It is known that \WL is a proper model of $L$.

\subsection{The translations and modal companions}

Let's define the G\"{o}del--Tarski translation
$\TGT \colon \Fmli \rarr \Fmlb$ by induction of the construction of a formula:
\begin{align*}
	\TGT(p) & := \Box p,
	&
	\TGT(\bot) & := \bot,
	\\
	\TGT(\varphi \wedge \psi) & := \TGT(\varphi) \wedge \TGT(\psi),
	&
	\TGT(\varphi \vee \psi) & := \TGT(\varphi) \vee \TGT(\psi),
	\\
	\TGT(\varphi \rarr \psi) & := \Box (\TGT(\varphi) \rarr \TGT(\psi)),
\end{align*}
where $p \in \mathrm{Prop}$.
Notice that the definition of operations in the algebra of open elements reproduces the definition of the translation \TGT for connectives
$\vee$, $\wedge$, $\rarr$, $\bot$. Thus, we have for every
$\varphi \in \Fmli$ (see, e.g., \cite[Lemma\,3.10]{GabMaks05})
\begin{equation}
	\GB \models \varphi
	\iff
	\bfB \models \TGT \varphi . \tag{Top}\label{Top}
\end{equation}

Now, let $L \in \EInt$. A logic $M \in \ESf$ is called a {\em modal companion of} $L$, if \TGT faithfully embeds $L$ into $M$: for every $\varphi \in \Fmli$
\[
	\varphi \in L \iff \TGT \varphi \in M.
\]

It is well-known that for each $L \in \EInt$, the \Sf-extensions
$\tau L := \Sf + \{ \TGT \varphi \mid \varphi \in L \}$ and
$\sigma L := \tau L + \Grz$ are the least (see \cite{DummettLemmon59}) and the greatest (see \cite{Blok76, Esakia76}) modal companions  of $L$ respectively.

We define the translation
$\Tb \colon \Fmls \rarr \Fmlbs$:{ \small
\begin{align*}
	\Tb(p) & := \Box p,
	&
	\Tb(\rsim p) & := \Box \rsim p,
	\\
	\Tb(\bot) & := \bot,
	&
	\Tb(\rsim \bot) & := \rsim \bot,
	\\
	\Tb(\varphi \wedge \psi) & := \Tb(\varphi) \wedge \Tb(\psi),
	&
	\Tb(\rsim (\varphi \wedge \psi)) & :=
				\Tb(\rsim \varphi) \vee \Tb(\rsim \psi),
	\\
	\Tb(\varphi \vee \psi) & := \Tb(\varphi) \vee \Tb(\psi),
	&
	\Tb(\rsim (\varphi \vee \psi)) & :=
				\Tb(\rsim \varphi) \wedge \Tb(\rsim \psi),
	\\
	\Tb(\varphi \rarr \psi) & := \Box (\Tb(\varphi) \rarr \Tb(\psi)),
	&
	\Tb(\rsim (\varphi \rarr \psi)) & :=
				\Tb(\varphi) \wedge \Tb(\rsim \psi).
\end{align*}}
It is easy to see that the restriction of \Tb to \Fmli coincides with \TGT. Moreover, \Tb moves strong negation through positive connectives and removes even iterations of strong negation, so \rsim occurs in $\Tb \varphi$ only in front of propositional variables and
$\bot$.

In \cite{OdWans2010} it was proved that \Tb faithfully embeds \Nf into \Bsf, so it is natural to define a {\em modal companion of} $L \in \ENf$ as a \Bsf-extension, in which $L$ is faithfully embedded by \Tb.

By analogy with \Int-extensions, for $L \in \ENf$ we can define the candidate
$\taub L := \Bsf + \{ \Tb \varphi \mid \varphi \in L \}$
to be the least modal companion of $L$. And it is not hard to see that if $L$ has some modal companion then $\taub L$ is the least modal companion of $L$.

\section{Preliminary algebraic observations}

The main goal of the present work, which gives motivation to the following algebraic results, is to explore which \Nf-extensions have modal companions.

Recall the idea how to get the results on modal companions of superintuitionistic logics:
\begin{enumerate}[1.]
	
	\item For a Heyting algebra \bfA
	take a \TBA \bfB such that $\GB = \bfA$.
	
	\item In view of \eqref{Top} we have that $\mathrm{L} \bfB$
	is a modal companion of $\mathrm{L} \bfA$. 
	
	\item Now just take \bfA to be some proper model of $L \in \EInt$.

\end{enumerate}

In the present section we provide a representaion of twist-structures over Heyting algebras as subsets of twist-structures over {\TBA}s defined in some regular way, for which we have the property of preserving the truth similar to \eqref{Top} but relative to the translation \Tb.

\subsection{From \Bsf-twist-structures to \Nf-twist-structures}

In this subsection we construct a twist-structure over a Heyting algebra starting from a twist-structure over a \TBA. Let's fix an arbitrary \TBA
$\bfB = \langle B; \vee, \wedge, \rarr, \bot, \Box \rangle$ with the greatest element $1$, $\nabla \in \FbB$ and $\Delta \in \IdiB$. We will consider the twist-structure
$\bfT := \TwBND$ throughout this subsection.

We will search for a twist-structure over the algebra of open elements \GB or its subalgebra. This motivates us to consider the following set:
\[
	\GsT : = \{ (a, b) \in \bfT \mid \Box a = a, \Box b = b \}.
\]
We define then $\GmT : = \pi_1 (\GsT)$. Also we consider a set
\[
	\Lnabla := \{ a \in \GsB \mid a \vee \Box \neg a \in \nabla\},
\]
the sense of which will be clear later.

\begin{lemma}\label{l311} The following hold:
	
	\begin{enumerate}[1.]
		
		\item $a \wedge \Box \neg a = \bot$ for all $a\in B$, so
		$a \wedge \Box \neg a \in \Delta$;
		
		\item\label{l3112} $\GmT \subseteq \GsB$;
		
		\item\label{l3113} \GsT is closed under
		$ \vee_{\GB^{\bowtie}},
		\wedge_{\GB^{\bowtie}}, \rsim_{\GB^{\bowtie}} $,
		and
		$\bot_{\GB^{\bowtie}} \in \GsT$;
		
		\item\label{l3114} $\pi_2 (\GsT) = \GmT $;
		
		\item \GmT is closed under $ \vee_{\GB}, \wedge_{\GB} $
		and $ \bot_{\GB} \in \GmT $;
		
		\item\label{l3116} $ \Lnabla \subseteq \GmT $;
		
		\item\label{l3117} \Lnabla is a subalgebra of \GB.
		
	\end{enumerate}
	
\end{lemma}

\begin{proof}
	
	Items 1, 2, 3 are obvious. Item 5 is a consequence of item 3. Item 4
	holds because \GsT is closed under $\rsim_{\GB^{\bowtie}}$.
	
	6. If $x \in \Lnabla$, then
	$x \vee \Box \neg x \in \nabla$ and $\Box x = x$. In view of (1)
	$x \wedge \Box \neg x \in \Delta$. So
	$(x, \Box \neg x) \in \bfT $ by definition of \TwBND. And
	$x \in \GmT$.
	
	7. Suppose $a, b \in \Lnabla$. So, obviously,
	\[
		a \vee_{\GB} b,\  a \wedge_{\GB} b,\  a \rarr_{\GB} b,\  \bot_{\GB}
		\in \GsB.
	\]
	
	For $\vee_{\GB}$ we have
	\begin{multline*}
		(a \vee_{\GB} b) \vee \Box \neg (a \vee_{\GB} b) =
		(a \vee b) \vee \Box \neg (a \vee b) = \\
		= (a \vee b) \vee \Box( \neg a \wedge \neg b ) =
		(a \vee b) \vee (\Box( \neg a) \wedge \Box( \neg b)) =\\
		= (a \vee \Box (\neg a) \vee b) \wedge (b \vee \Box (\neg b) \vee a) 
		\geqslant
		(a \vee \Box \neg a) \wedge (b \vee \Box \neg b) \in \nabla.
	\end{multline*}
	
	The case of $\wedge_{\GB}$ is considered as follows:
	\begin{multline*}
		(a \wedge_{\GB} b) \vee \Box \neg (a \wedge_{\GB} b) =
		(a \wedge b) \vee \Box \neg (a \wedge b) = \\
		= (a \wedge b) \vee \Box( \neg a \vee \neg b )
		\geqslant
		(a \wedge b) \vee (\Box( \neg a) \vee \Box( \neg b)) =\\
		= (a \vee \Box (\neg a) \vee \Box(\neg b)) \wedge (b \vee \Box (\neg b) \vee \Box(\neg a))
		\geqslant
		(a \vee \Box \neg a) \wedge (b \vee \Box \neg b) \in \nabla.
	\end{multline*}
	
	Finally, consider the implication:
	\begin{multline*}
		(a \rarr_{\GB} b) \vee \Box \neg (a \rarr_{\GB} b) =
		\Box (a \rarr b) \vee \Box \neg \Box (a \rarr b) = \\
		= \Box (\neg a \vee b) \vee \Box \Diamond \neg (\neg a \vee b) =
		\Box (\neg a \vee b) \vee \Box \Diamond (a \wedge \neg b)
		\geqslant \\
		\geqslant
		(\Box (\neg a) \vee \Box b) \vee \Box(a \wedge \neg b) =
		\Box (\neg a) \vee b \vee (a \wedge \Box \neg b) =\\
		= (a \vee \Box (\neg a) \vee b) \wedge (b \vee \Box (\neg b) \vee \Box(\neg a))
		\geqslant
		(a \vee \Box \neg a) \wedge (b \vee \Box \neg b) \in \nabla.
		\end{multline*}
	
	Obviously,
	\[
		\bot_{\GB} \vee \Box \neg \bot_{\GB} =
		\bot \vee \Box \neg \bot = 1 \in \nabla. \qedhere
	\]
	
\end{proof}

We will explore when \GsT is the domain of some twist-structure over a Heyting algebra. So it is important to consider its invariants. Define the sets:
\begin{align*}
	\nabg := \{ a \vee b \mid (a, b) \in \GsT \},\\
	\Delg := \{ a \wedge b \mid (a, b) \in \GsT \}.
\end{align*}
The following lemma gives other descriptions of these sets.

\begin{lemma}\label{l312}\hfill
	
	\begin{enumerate}[1.]
		
		\item $\nabg = \nabla \cap \GsB = \nabla \cap \GmT =
		\nabla \cap \Lnabla$.
		
		\item $\Delg = \Delta \cap \GsB = \Delta \cap \GmT$.
		
	\end{enumerate}
	
\end{lemma}

\begin{proof}
	
	1. Let's prove $\nabla \cap \GsB \subseteq \nabg$. If
	$x \in \nabla \cap \GsB$, then $x \vee \bot = x \in \nabla$ and
	$x \wedge \bot = \bot \in \Delta$, so $(x, \bot) \in \bfT$.
	Moreover, $(x, \bot) \in \GsT $, since $\Box x = x$.
	Thus, $x \vee \bot = x \in \nabg$.
	
	Let's show that $\nabg \subseteq \nabla \cap \Lnabla$. If $x \in \nabg$,
	then $x = a \vee b$ for $(a, b) \in \GsT$. Obviously, $\Box x = x$
	and $x = a \vee b \in \nabla$, so $x \vee \Box \neg x \in \nabla$
	and $x \in \Lnabla$.
	
	Items \ref{l3112} and \ref{l3116} of Lemma~\ref{l311} imply:
	\[
			\nabg
		\subseteq
			\nabla \cap \Lnabla
		\subseteq
			\nabla \cap \GmT
		\subseteq
			\nabla \cap \GsB
		\subseteq
			\nabg.
	\]
	
	2. Let's prove $\Delg \subseteq \Delta \cap \GmT$. If
	$x \in \Delg$, then $x = a \wedge b$ for $(a, b) \in \GsT$.
	Firstly, $x = a \wedge b \in \Delta$.
	Secondly, item \ref{l3113} of Lemma~\ref{l311} implies
	$(a, b) \wedge_{\bfB^{\bowtie}} \rsim_{\bfB^{\bowtie}} (a, b) =
	(a \wedge b, a \vee b) \in \GsT$, so $x = a \wedge b \in \GmT$.
	
	Let's show that $\Delta \cap \GsB \subseteq \Delg $. If
	$x \in \Delta \cap \GsB$, then $x \vee 1 = 1 \in \nabla$,
	$x \wedge 1 = x \in \Delta$, so $(x, 1) \in \bfT$. Moreover,
	$(x, 1) \in \GsT $, because $\Box x = x$, so
	$x \wedge 1 = x \in \Delg$.
	
	From item \ref{l3112} of Lemma~\ref{l311} we have
	\[
			\Delg
		\subseteq
			\Delta \cap \GmT
		\subseteq
			\Delta \cap \GsB
		\subseteq
			\Delg. \qedhere
	\]
	
\end{proof}

The following lemma will be useful in Subsection~\ref{ss33}.

\begin{lemma}\label{l313}
	
	$\neg_{\GB} \neg_{\GB} a \in \Delg$ for all $a \in \Delg$.
	
\end{lemma}

\begin{proof}
	
	In view of Lemma~\ref{l312}, $\Delg = \Delta \cap \GsB $. If
	$a \in \Delg$, then, obviously,
	$\neg_{\GB} \neg_{\GB} a \in \GsB$. We have $ a \in \Delta$, and then
	$ \Diamond a \in \Delta $, since $\Delta \in \IdiB$.
	Since $\Box \Diamond a \leql \Diamond a$, we obtain
	\[
		\Box \Diamond a =
		\Box \neg \Box \neg a =
		\neg_{\GB} \neg_{\GB} a \in \Delta. \qedhere
	\]
	
\end{proof}

The following lemma gives us conditions for \GsT to be the domain of some twist-structure over a subalgebra of \GB.

\begin{lemma}\label{l314}

	\[
		\GmT \subseteq \Lnabla \iff \GsT \text{ is closed under
		$\rarr_{\GB^{\bowtie}}$.}
	\]
	
\end{lemma}

\begin{proof}
	
	$\Longrightarrow$. Let $(a, b),(c, d) \in \GsT$. We know that $c \wedge d \in \Delta$, so $\Box (a \rarr c) \wedge (a \wedge d) \in \Delta$ can be shown with ease:
	\begin{multline*}
		\Box (a \rarr c) \wedge (a \wedge d)
		\leql
		(a \rarr c) \wedge (a \wedge d) = \\
		= (a \wedge (a \rarr c)) \wedge d =
		a \wedge c \wedge d
		\leql
		c \wedge d \in \Delta.
	\end{multline*}
	
	Notice that
	\[
		\Box (a \rarr c) \vee (a \wedge d) =
		(\Box (a \rarr c) \vee a)
		\wedge
		(\Box (a \rarr c) \vee d),
	\]
	and we can easily show $ \Box (a \rarr c) \vee d \in \nabla$:
	\[
		\Box (a \rarr c) \vee d = \Box (\neg a \vee c) \vee d
		\geqslant
		\Box c \vee d = c \vee d \in \nabla.
	\]
	
	In view of $a \in \GmT$ and $\GmT \subseteq \Lnabla$ we have
	$\Box (a \rarr c) \vee a \in \nabla$:
	\[
		\Box (a \rarr c) \vee a = \Box (\neg a \vee c) \vee a
		\geqslant
		\Box (\neg a) \vee a \in \nabla.
	\]
	
	So $\Box (a \rarr c) \vee (a \wedge d) \in \nabla$, and from 
	$a \wedge d \in \GsB$ we have thus proved
	\[
		(a, b) \rarr_{\GB^{\bowtie}} (c, d) =
		(\Box (a \rarr c), a \wedge d) \in \GsT.
	\]
	
	$\Longleftarrow$. If $a \in \GmT$, then $(a, b) \in \GsT$ for
	some $b \in \GsB$. We know that
	$\bot_{\GB^{\bowtie}} = (\bot, 1) \in \GsT$, so, since
	\GsT is closed under $\rarr_{\GB^{\bowtie}}$,
	\[
		(a, b) \rarr_{\GB^{\bowtie}} (\bot, 1) =
		(\Box (a \rarr \bot), a \wedge 1) =
		(\Box \neg a, a) \in \GsT,
	\]
	which implies $ a \vee \Box \neg a \in \nabla $. Since $\Box a = a$,
	we get $a \in \Lnabla$. \qedhere
	
\end{proof}

And now we explain how to extract a twist-structure over a Heyting algebra from a twist-structure over a \TBA.

\begin{proposition} 
	
	Let $\GmT = \Lnabla$. Then:
	\begin{enumerate}[1.]
		
		\item $ \boldsymbol{\Gamma}(\bfT) = \langle \GmT;
		\vee_{\GB}, \wedge_{\GB}, \rarr_{\GB}, \bot_{\GB} \rangle $
		is a subalgebra of \GB.
		
		\item $\GT = \langle \GsT;
		\vee_{\GB^{\bowtie}}, \wedge_{\GB^{\bowtie}},
		\rarr_{\GB^{\bowtie}}, \bot_{\GB^{\bowtie}},
		\rsim_{\GB^{\bowtie}} \rangle$
		is a twist-structure over the Heyting algebra
		$\boldsymbol{\Gamma}(\bfT)$.
		More precisely,
		\[
			\GT
			= Tw(\boldsymbol{\Gamma}(\bfT), \nabg, \Delg).
		\]
		
	\end{enumerate}
	
\end{proposition}

\begin{proof}
	
	1. This is just a consequence of item \ref{l3117} of Lemma~\ref{l311}.
	
	2. In view of item~\ref{l3114} of Lemma~\ref{l311} $\GsT \subseteq
	\pi_1 (\GsT) \times \pi_2 (\GsT) =
	\GmT \times \GmT$, and $\pi_1 (\GsT) = \GmT$ by definition. Obviously,
	the operations of $\GB^{\bowtie}$ coincide with operations of
	$\boldsymbol{\Gamma}(\bfT)^{\bowtie}$. The condition $\GmT = \Lnabla$
	implies by
	Lemma~\ref{l314} that
	\GsT is closed under $\rarr_{\GB^{\bowtie}}$. In view of item
	\ref{l3113} of Lemma~\ref{l311} \GsT is
	closed under the operations $\vee_{\GB^{\bowtie}}$,
	$\wedge_{\GB^{\bowtie}}$, $\rsim_{\GB^{\bowtie}}$ and
	$\bot_{\GB^{\bowtie}} \in \GsT$. So we have that
	$\GT = \langle \GsT; \vee_{\GB^{\bowtie}}, \wedge_{\GB^{\bowtie}},
	\rarr_{\GB^{\bowtie}}, \bot_{\GB^{\bowtie}}, \rsim_{\GB^{\bowtie}}
	\rangle$ is a twist-structure over $\boldsymbol{\Gamma}(\bfT)$. From
	$a \vee_{\GB} b = a \vee b$ and $a \wedge_{\GB} b = a \wedge b$ we get:
	\[
		\GT
		=
		Tw(\boldsymbol{\Gamma}(\bfT), \nabg, \Delg). \qedhere
	\]
	
\end{proof}

\begin{definition}
	
	In case $\GmT = \Lnabla$ we call the twist-structure \GT\
	the {\em algebra of open pairs of the twist-structure} \bfT.
	
\end{definition}

\subsection{Preserving the truth}

We work with the same \TBA \bfB, \mbox{$\nabla \in \FbB$,} $\Delta \in \IdiB$ and $\bfT := \TwBND$.

The main goal of the present subsection is to find conditions guaranteeing that for every $\varphi \in \Fmls$ the following holds:
\begin{equation}
	\GT \models \varphi
	\iff
	\bfT \models \Tb \varphi. \tag{TwTop}\label{TwTop}
\end{equation}

Assume that $\GmT = \Lnabla$ and
$v \colon \mathrm{Prop} \rarr \GsT$. Then functions
$h_{1}(v) \colon \Fmlbs \rarr \bfT$ and
$h_{2}(v) \colon \Fmls \rarr \GT$ denote unique homomorphisms such that
$h_{1}(v) \upharpoonright_{\mathrm{Prop}} = v$ and
$h_{2}(v) \upharpoonright_{\mathrm{Prop}} = v$. Also, for each map
$w \colon \mathrm{Prop} \rarr \bfT$, we denote by
$h_{1}(w)$ the unique \bfT-valuation such that
$h_{1}(w) \upharpoonright_{\mathrm{Prop}} = w$.

\begin{lemma}\label{l321}
	
	Suppose that $\GmT = \Lnabla$ and
	$v \colon \mathrm{Prop} \rarr \GsT$. Then for every
	$\varphi \in \Fmls$ it is true that
	$\pi_{1}( h_{1}(v)( \Tb \varphi ) ) = \pi_{1}( h_{2}(v)( \varphi ))$.
	
\end{lemma}

\begin{proof}
	
	It is a long but routine work to check this by induction of the
	construction of a formula. For instance, if
	$\pi_{1}( h_{1}(v)( \Tb \varphi ) ) = \pi_{1}( h_{2}(v)( \varphi ) )$
	and
$\pi_{1}( h_{1}(v)( \Tb \rsim \psi ) ) = \pi_{1}( h_{2}(v)( \rsim \psi ) )$,
	then
	\begin{multline*}
		\pi_{1}(
			h_{1}(v)(
				\Tb( \rsim (\varphi \rarr \psi) )
			)
		)
		=
		\pi_{1}(
			h_{1}(v)(
				\Tb( \varphi ) \wedge \Tb( \rsim \psi )
			)
		) = \\
		=	
		\pi_{1}(
			h_{1}(v)(
				\Tb \varphi 
			)
		)
		\wedge_{\bfB}
		\pi_{1}(
			h_{1}(v)(
				\Tb \rsim \psi 
			)
		) = \\
		=
		\pi_{1}(
			h_{2}(v)(
				\varphi
			)
		)
		\wedge_{\GB}
		\pi_{1}(
			h_{2}(v)(
				\rsim \psi
			)
		) = \\
		=
		\pi_{1}(
			h_{2}(v)(
				\varphi
			)
		)
		\wedge_{\GB}
		\pi_{1}(
			\rsim_{\boldsymbol{\Gamma}(\bfT)^{\bowtie}} h_{2}(v)(
				\psi
			)
		) = \\
		=
		\pi_{1}(
			h_{2}(v)(
				\varphi
			)
		)
		\wedge_{\GB}
		\pi_{2}(
			h_{2}(v)(
				\psi
			)
		) = \\
		=
		\pi_{2}(	
			h_{2}(v)( \varphi )
			\rarr_{\boldsymbol{\Gamma}(\bfT)^{\bowtie}}
			h_{2}(v)( \psi )
		) = \\
		=
		\pi_{1}( \rsim_{\boldsymbol{\Gamma}(\bfT)^{\bowtie}}
			(
				h_{2}(v)( \varphi )
				\rarr_{\boldsymbol{\Gamma}(\bfT)^{\bowtie}}
				h_{2}(v)( \psi )
			)
		)
		=
		\pi_{1}(
			h_{2}(v)(
				\rsim (\varphi \rarr \psi)
			)
		). \qedhere
	\end{multline*}
	
\end{proof}

\begin{proposition}\label{p322}
	
	Suppose that $(\Box a, \Box b) \in \bfT$ for all $(a, b) \in \bfT$. Then
	$ \GsB = \GmT = \Lnabla $. Furthermore, for every $\varphi \in \Fmls$,
	\[
		\GT \models \varphi \iff \bfT \models \Tb \varphi . \tag{TwTop}
	\]
	
\end{proposition}

\begin{proof}
	
	Let us show that $\GsB \subseteq \Lnabla$. If $x \in \GsB$, then
	$(x, \neg x) \in \bfT$, because $x \wedge \neg x = \bot \in \Delta$,
	$x \vee \neg x = 1 \in \nabla$. So
	$(\Box x, \Box \neg x) = (x, \Box \neg x) \in \bfT$, which
	implies $x \vee \Box \neg x \in \nabla$. Thus, $x \in \Lnabla$.
	
	Now, we prove the equivalence \eqref{TwTop}.
	
	$\Longleftarrow$. Let $\varphi \in \Fmls$ and
	$\bfT \models \Tb \varphi$. Consider an arbitrary map
	$v \colon \mathrm{Prop} \rarr \GsT$. By Lemma~\ref{l321}, we have
	\[\pi_{1}( h_{2}(v)( \varphi ) ) = \pi_{1}( h_{1}(v)( \Tb \varphi ) ) =
	1_{\bfB} = 1_{\GB}.\]
	
	$\Longrightarrow$. Let $\varphi \in \Fmls$ and
	$\GT \models \varphi$. Consider an arbitrary map
	$v \colon \mathrm{Prop} \rarr \bfT$. Define
	$w \colon \mathrm{Prop} \rarr \GsT$ by
	\[
		w(p) := ( \Box \pi_{1}(v(p)), \Box \pi_{2}(v(p)) ).
	\]
	The map $w$ is well-defined, since $(\Box a, \Box b) \in \bfT$ for all
	$(a, b) \in \bfT$.
	
	We know that \rsim occurs in $\Tb \varphi$ only in front of
	propositional variables and $\bot$. From $\Tb(p) = \Box p$ and
	$\Tb(\rsim p) = \Box \rsim p$ we have that there is a $\Box$ before each
	occurrence of $p$ or $\rsim p$. Thus, if
	$\varphi = \varphi(p_{1}, \dots, p_{n})$, then there is a formula
	$\psi(p_{1}, \dots, p_{n}, q_{1}, \dots, q_{n}, s)
	\in \Fmlb$ such that
	$\Tb \varphi = \psi(\Box p_{1}, \dots, \Box p_{n},
	\Box \rsim p_{1}, \dots, \Box \rsim p_{n}, \rsim \bot)$. So we have
	\begin{multline*}
		\pi_{1}(h_{1}(v)(\Tb \varphi)) = \\
		=
		\pi_{1}(
			h_{1}(v)(
				\psi(
					\Box p_{1}, \dots, \Box p_{n},
					\Box \rsim p_{1}, \dots, \Box \rsim p_{n},
					\rsim \bot
				)
			)
		) = \\
		=
		\psi(
		\Box \pi_{1}(v(p_{1})), \dots, \Box \pi_{1}(v(p_{n})),
		\Box \pi_{2}(v(p_{1})), \dots, \Box \pi_{2}(v(p_{n})),
		1
		) = \\
		=
		\psi(
		\Box \Box \pi_{1}(v(p_{1})), \dots, \Box \Box \pi_{1}(v(p_{n})),
		\Box \Box \pi_{2}(v(p_{1})), \dots, \Box \Box \pi_{2}(v(p_{n})),
		1
		) = \\
		=
		\psi(
		\Box \pi_{1}(w(p_{1})), \dots, \Box \pi_{1}(w(p_{n})),
		\Box \pi_{2}(w(p_{1})), \dots, \Box \pi_{2}(w(p_{n})),
		1
		) = \\
		=
		\pi_{1}(
			h_{1}(w)(
				\psi(
					\Box p_{1}, \dots, \Box p_{n},
					\Box \rsim p_{1}, \dots, \Box \rsim p_{n},
					\rsim \bot
				)
			)
		) = \\
		=
		\pi_{1}(h_{1}(w)(\Tb \varphi))
		=
		\pi_{1}(h_{2}(w)(\varphi)) = 1_{\GB} = 1_{\bfB}. \qedhere
	\end{multline*}
	
\end{proof}

Now we recall the Kripke semantics for \Grz.

A {\em frame} is a pair $\mathcal{W} = \langle W, \leql \rangle$, where $W$ is a non-empty set of {\em worlds}, and $\leql$ is a partial ordering on $W$. 

A {\em model} over $\mathcal{W} = \langle W, \leql \rangle$ is a pair
$\mathcal{M} = \langle \mathcal{W}, v \rangle$, where
$v \colon \mathrm{Prop} \rarr 2^{W}$ is a {\em valuation} of propositional variables.

Given a model $\mathcal{M} = \langle \mathcal{W}, v \rangle$ over
$\mathcal{W} = \langle W, \leql \rangle$ we define the truth relation
between formulas of \Fmlb and worlds of $W$:
\begin{itemize}
	\item[--] $\mathcal{M}, x \models p \iff x \in v(p)$\,; 
	
	\item[--]  $\mathcal{M}, x \not\models \bot$\,; 
	
	\item[--]  $\mathcal{M}, x \models {\varphi \wedge \psi} 
	\iff 
	(\mathcal{M}, x \models \varphi \text{ and } \mathcal{M}, x \models \psi)$\,; 
	
	\item[--]  $\mathcal{M}, x \models {\varphi \vee \psi}
	\iff
	(\mathcal{M}, x \models \varphi \text{ or } \mathcal{M}, x \models \psi)$\,; 
	
	\item[--]  $\mathcal{M}, x \models {\varphi \rarr \psi}
	\iff
	(\mathcal{M}, x \models \varphi\ \Rightarrow\ \mathcal{M}, x \models \psi)$\,; 
	
	\item[--]  $\mathcal{M}, x \models {\Box \varphi} \iff  
	{\forall
		{y \in W}}\, \left( {x \leql y}\ \Rightarrow\ {\mathcal{M}, y \models \varphi} \right)$.
\end{itemize}
It is not hard to see that
\[
	\mathcal{M}, x \models {\Diamond \varphi} \iff  
	{\exists
		{y \in W}}\, \left( {x \leql y} \text{ and }
		{\mathcal{M}, y \models \varphi} \right).
\]

A formula $\varphi$ is {\em true} in a model $\mathcal{M}$,
$\mathcal{M} \models \varphi$, if $\mathcal{M}, x \models \varphi$ for all
$x \in W$; $\varphi$ is {\em valid} in a frame $\mathcal{W}$,
$\mathcal{W} \models \varphi$, if $\mathcal{M} \models \varphi$ for each model $\mathcal{M}$ over $\mathcal{W}$.

It is well-known (see, for example, \cite[Theorem\,5.51]{ChagrovZ}) that for every $\varphi \in \mathrm{Fm}_{\langu^{\Box}}$:
\[
	\varphi \in \Grz \iff \mathcal{W} \models \varphi \text{ for all finite 	frames $\mathcal{W}$}.
\]

\begin{lemma}\label{l323}
	
	\[
		\big[
			\Box(p \vee q)
			\wedge
			(\Box p \vee \Box \Diamond \neg p)
			\wedge
			(\Box q \vee \Box \Diamond \neg q)
		\big]
		\rarr (\Box p \vee \Box q) \in \Grz.
	\]
	
\end{lemma}

\begin{proof}
	
	Let $\mathcal{W}$ be an arbitrary finite frame, $\mathcal{M}$ be a model
	over $\mathcal{W}$ and $x \in W$.
	
	We need to show that if
	\begin{enumerate}
		
		\item $\mathcal{M}, x \models \Box(p \vee q)$,
		
		\item $\mathcal{M}, x \models \Box p \vee \Box \Diamond \neg p$ and
		
		\item $\mathcal{M}, x \models \Box q \vee \Box \Diamond \neg q$,
		
	\end{enumerate}
	then $\mathcal{M}, x \models \Box p \vee \Box q$.
	
	Suppose that $\mathcal{M}, x \not\models \Box p $ and
	$\mathcal{M}, x \not\models \Box q $. In view of (2) and (3)
	it means that
	$\mathcal{M}, x \models \Box \Diamond \neg p$ and
	$\mathcal{M}, x \models \Box \Diamond \neg q$.
	
	There is a maximal $y \in W$ w.r.t. \leql such that
	$x \leql y$, because $W$ is finite. From
	$\mathcal{M}, x \models \Box \Diamond \neg p$ and
	$\mathcal{M}, x \models \Box \Diamond \neg q$ we get that
	$\mathcal{M}, y \models \Diamond \neg p$ and
	$\mathcal{M}, y \models \Diamond \neg q$.
	
	From the maximality of $y$ we have
	$\mathcal{M}, y \models \neg p $ and
	$\mathcal{M}, y \models \neg q $. But we also have
	$\mathcal{M}, y \models p \vee q $ because of (1): a contradiction.
	Thus, we have proved
	$\mathcal{M}, x \models \Box p \vee \Box q$. \qedhere
	
\end{proof}

\begin{lemma}\label{l324}
	
	Suppose that $\bfB \models \Grz$ and $\GsB = \Lnabla$. Then
	$(\Box a, \Box b) \in \bfT$ for all $(a, b) \in \bfT$.
	
\end{lemma}

\begin{proof}
	
	Suppose that $(a, b) \in \bfT$. Then $a \vee b \in \nabla$ and
	$a \wedge b \in \Delta$. We can easily show that
	$\Box a \wedge \Box b \in \Delta$:
	\[
		\Box a \wedge \Box b \leql a \wedge b \in \Delta.
	\]
	
	Since
	$\nabla$ is an open filter, $\Box (a \vee b) \in \nabla$ holds.
	We have $\Box a \vee \Box \neg \Box a = \Box a \vee \Box \Diamond
	\neg a$
	and
	$\Box b \vee \Box \neg \Box b = \Box b \vee \Box \Diamond \neg b$,
	so in view of $\GsB = \Lnabla$
	\[
		\Box(a \vee b)
		\wedge
		(\Box a \vee \Box \Diamond \neg a)
		\wedge
		(\Box b \vee \Box \Diamond \neg b) \in \nabla.
	\]
	From $\bfB \models \Grz$ and Lemma~\ref{l323} we obtain that
	$\Box a \vee \Box b \in \nabla$. Since
	$\Box a \wedge \Box b \in \Delta$, we get
	$(\Box a, \Box b) \in \bfT$. \qedhere
	
\end{proof}

\begin{theorem}\label{t325}
	
	Suppose that $\bfB \models \Grz$ and $\GsB = \Lnabla$. Then for every
	$\varphi \in \Fmls$
	\[
		\GT \models \varphi \iff \bfT \models \Tb \varphi.
	\]
	
\end{theorem}

\begin{proof}
	
	This is an immediate consequence of Lemma~\ref{l324} and
	Proposition~\ref{p322}. \qedhere
	
\end{proof}

\subsection{From \Nf-twist-structures to \Bsf-twist-structures}\label{ss33}

In this subsection given a twist-structure \clA over a Heyting algebra we construct a twist-structure over a \TBA such that \clA is a subalgebra of its algebra of open pairs.

We fix an arbitrary Heyting algebra
$\bfA = \langle A; \vee_{\bfA}, \wedge_{\bfA}, \rarr_{\bfA}, \bot_{\bfA} \rangle$ with the greatest element $1_{\bfA}$, and we will search for a twist-structure over a \TBA $\bfB := s(\bfA)$. We write the operations of \bfA and \bfB without subscripts if it does not lead to a confusion. In what follows we always write $\neg$ for the negation operation of \bfA and $\neg_{\bfB}$ for the negation operation of \bfB.

It is well-known (see, for example, \cite[pp.~61--62]{GabMaks05}) that there is a bijective correspondence between open filters of \bfB and filters of \GB given by the following mutually inverse mappings:
\begin{align*}
	\delta & \colon \nabla \in \FbB
	\longmapsto
	\delta(\nabla) = \nabla \cap \GsB \in \mathcal{F}(\GB) \\
	\rho & \colon \nabla' \in \mathcal{F}(\GB)
	\longmapsto
	\rho(\nabla') = \{ x \in B \mid \Box x \in \nabla' \} \in \FbB
\end{align*}

For each $\Delta \in \IA$ there is the least closed ideal of the \TBA \bfB containing $\Delta$, which can be defined as
\[
	\sigma(\Delta) := \{ x \in B \mid x \leql_{\bfB} \Diamond y
	\text{ for some $y \in \Delta$} \}.
\]
It is easy to check that $\sigma(\Delta) \in \IdiB$, and it is obvious that $\sigma(\Delta)$ is contained in each closed ideal of \bfB containing
$\Delta$.

We call an ideal $\Delta$ of a Heyting algebra \bfA\ {\em closed}, if $\neg \neg a \in \Delta$ for all $a \in \Delta$. Not each ideal of \bfA is closed. For instance, an ideal $\{ x \in A \mid x \leql_{\bfA} a \}$ with $a \not= \neg\neg a$ is not closed. We can see from Lemma~\ref{l313} that the ideal
\Delg of the algebra \GT of open pairs is closed.

It is not hard to check that for $\Delta \in \IA$ the set
\[
	N(\Delta) := \{ a \in A \mid 
		a \leql_{\bfA} \neg \neg b \text{\;  for some $b \in \Delta$} \}.
\]
is a closed ideal of \bfA, and it is a subset of each closed ideal of \bfA containing $\Delta$, so it is the least closed ideal of \bfA containing
$\Delta$. We call this set the {\em closure of} $\Delta$.

\begin{lemma}\label{l331}
	
	$\sigma(\Delta) \cap A = N(\Delta)$.
	
\end{lemma}

\begin{proof}
	
	$\subseteq$. If $x \in \sigma(\Delta) \cap A$, then $\Box x = x$ and
	$x \leql_{\bfB} \Diamond y$ for some $y \in \Delta$. So
	$x = \Box x \leql_{\bfB} \Box \Diamond y = \neg \neg y$, which means
	$x \in N(\Delta)$.
	
	$\supseteq$. Let $x \in N(\Delta)$. Then $x \in A$ and $x \leql_{\bfA}
	\neg \neg y$ for some $y \in \Delta$. So
	$x \leql_{\bfB} \Box \Diamond y \leql_{\bfB} \Diamond y$,
	which means $x \in \sigma(\Delta)$. \qedhere
	
\end{proof}

Now we fix $\nabla \in \FdA$ and $\Delta \in \IA$.

\begin{theorem}\label{t332}
	
	Let $\bfT = Tw(\bfB, \rho(\nabla), \sigma(\Delta))$. Then for
	every $\varphi \in \Fmls$
	\[
		\bfT \models \Tb \varphi
		\iff
		Tw(\bfA, \nabla, N(\Delta)) \models \varphi.
	\]
	
\end{theorem}

\begin{proof}
	
	We know that $\nabla \in \FdA$, so for every $a \in A$ we have
	$a \vee \neg a =
	a \vee \Box \neg_{\bfB} a \in \nabla \subseteq \rho(\nabla)$.
	In view of $A = \GsB$ we obtain $\GsB =
	\Lambda(\bfB, \rho(\nabla))$.
	
	It is well-known that if
	$\mathbf{C}$ is a \TBA such that
	$\mathbf{C}^{s} = \mathbf{C}$, then $\mathbf{C} \models \Grz$
	(see, for example, \cite[Corollary\,3.5.14]{EsakiaDual}).
	So $\bfB = s(\bfA) \models \Grz$.
	
	Now, let us show that $\GT = Tw(\bfA, \nabla, N(\Delta))$. Firstly,
	$A = \GsB = \GmT = \Lambda(\bfB, \rho(\nabla))$,
	because $\GsB = \Lambda(\bfB, \rho(\nabla))$. So
	\GT is a twist-structure over \bfA. By Lemma~\ref{l312}
	we know that the invariants of this twist-structure are
	$\rho(\nabla) \cap A
	 = \delta(\rho(\nabla)) = \nabla$ and $\sigma(\Delta) \cap A =
	 N(\Delta)$ (by Lemma~\ref{l331}). Thus,
	$\GT = Tw(\bfA, \nabla, N(\Delta))$.
	
	And now, the conclusion of the theorem follows from
	$\bfB \models \Grz$, $\GsB = \Lambda(\bfB,
	\rho(\nabla))$ and Theorem~\ref{t325}. \qedhere
	
\end{proof}

From Lemma~\ref{l313} and the proof of Theorem~\ref{t332} we infer

\begin{corollary}
	
	A twist-structure over a Heyting algebra can be represented as the
	algebra of open pairs of some twist-structure over a \TBA if and only if
	its ideal is closed.
	
\end{corollary}

\section{Modal companions}

\subsection{Logics with modal companions}

An easy corollary of Theorem~\ref{t332} is the following

\begin{proposition}\label{p411}
	
 	If $L \in \ENf$ has a proper model with closed ideal, then $\taub L$ is
 	a modal companion of $L$.
	
\end{proposition}

\begin{proof}
	
	Let $Tw(\bfA, \nabla, \Delta)$ be a proper model of $L$, and
	assume that $N(\Delta) = \Delta$. Let \bfT denote
	$Tw(s(\bfA), \rho(\nabla), \sigma(\Delta))$. By Theorem~\ref{t332},
	for every $\varphi \in \Fmls$ we have
	\[
		\bfT \models \Tb \varphi
		\iff
		Tw(\bfA, \nabla, \Delta) \models \varphi.
	\]
	
	So $\bfT \models \taub L$, and if $\varphi \not\in L$, then
	$\bfT \not\models \Tb \varphi$ implies $\Tb \varphi \not\in
	\taub L$. \qedhere
	
\end{proof}

We recall that the Lindenbaum twist-structure \WL over the Heyting algebra \TL is a proper model of $L \in \ENf$, and its invariants are
\[
	\nablaL = \{ [\varphi \vee \rsim \varphi]_{L} \mid 
			\varphi \in \Fmls \}
\]
and
\[
	\DeltaL = \{ [\varphi \wedge \rsim \varphi]_{L} \mid 
			\varphi \in \Fmls \}.
\]
In view of the previous proposition we are curious for which $L \in \ENf$ the ideal \DeltaL is closed.

\begin{proposition}
	
	Let $L \in \ENf$. Then
	\[
	\DeltaL \text{ is closed }
	\iff
	\neg \neg (p \wedge \rsim p) \leftrightarrow (p \wedge \rsim p) \in L.
	\] 
	
\end{proposition}

\begin{proof} $\Longleftarrow$. This is obvious.
	
	$\Longrightarrow$. If $\Delta_{L}$ is closed, then
	$\neg \neg (p_{1} \wedge \rsim p_{1})
	\leftrightarrow (\psi \wedge \rsim \psi) \in L$
	for some \mbox{$\psi \in \Fmls$}. By the substitution rule \mbox{(SUB)}
	we can consider $\psi = \psi (p_{1})$ such that it depends
	only on $p_{1}$.
	
	Define the following ideal in \TL:
	\[
		\Delta_{1} := \{ a \in \TL \mid a \leql_{\TL}
		[p_{1} \wedge \rsim p_{1}]_{\equiv_{L}} \}.
	\]
	
	Let $\mathcal{C}$ denote $Tw(\mathcal{T}(L), \nabla_{L}, \Delta_{1})$.
	The algebra $\mathcal{C}$ is a
	subalgebra of $Tw(\mathcal{T}(L), \nabla_{L}, \Delta_{L})$, so if
	\[Tw(\mathcal{T}(L), \nabla_{L}, \Delta_{L}) \models
	\neg \neg (p_{1} \wedge \rsim p_{1})
	\leftrightarrow (\psi \wedge \rsim \psi),\]
	then
	\[\mathcal{C} \models
	\neg \neg (p_{1} \wedge \rsim p_{1})
	\leftrightarrow (\psi \wedge \rsim \psi).\]
	
	Consider the $\mathcal{C}$-valuation $v(p_{1}) :=
	( [p_{1}]_{\equiv_{L}}, [\rsim p_{1}]_{\equiv_{L}} )$ and assume that
	$v(\psi(p_1)) = (a, b) $ for some $a, b \in \TL$.
	We have
	\begin{multline*}
		\pi_{1}(v(
		\neg \neg (p_{1} \wedge \rsim p_{1})
		\leftrightarrow
		(\psi \wedge \rsim \psi)
		)) = \\
		=
		\neg \neg [p_{1} \wedge \rsim p_{1}]_{\equiv_{L}}
		\leftrightarrow
		(a \wedge b) = 1_{\TL},
	\end{multline*}
	so $\neg \neg [p_{1} \wedge \rsim p_{1}]_{\equiv_{L}} = 
	(a \wedge b)$. From $a \wedge b \in \Delta_{1}$ we get
	\[
		\neg \neg [p_{1} \wedge \rsim p_{1}]_{\equiv_{L}} \in \Delta_{1},
	\]
	or
	\[
		\neg \neg [p_{1} \wedge \rsim p_{1}]_{\equiv_{L}}
		\leql_{\TL} [p_{1} \wedge \rsim p_{1}]_{\equiv_{L}},
	\]
	which implies $\neg \neg (p_{1} \wedge \rsim p_{1})
	\rarr
	(p_{1} \wedge \rsim p_{1}) \in L.$
	From $q \rarr \neg \neg q \in \Int$ we obtain
	$\neg \neg (p_{1} \wedge \rsim p_{1})
	\leftrightarrow
	(p_{1} \wedge \rsim p_{1}) \in L$. \qedhere
	
\end{proof}

\begin{theorem}\label{t413}
	
	Let $L$ be a logic that extends
	\[\Nf + \{ 
	\neg \neg (p \wedge \rsim p)
	\leftrightarrow
	(p \wedge \rsim p)\}.\] Then $\taub L$ is a
	modal companion of $L$.
	 
\end{theorem}

\begin{proof}
	
	According to the previous proposition, \WL is a proper model of $L$
	with closed ideal. Now, the conclusion follows from
	Proposition~\ref{p411}. \qedhere
	
\end{proof}

Theorem~\ref{t413} immediately implies

\begin{corollary}
	
	Each extension of the explosive Nelson's logic
	$\mathsf{N3}^{\bot} = \Nf +
	\{ \neg (p \wedge \rsim p) \}$ has a modal companion.
	
\end{corollary}

Now we strengthen the results of \cite{OdinVish24} by placing restrictions on the axioms of a logic that will guarantee that the logic has a modal companion.

\begin{lemma}\label{l414}
	
	Let $\bfA$ be a Heyting algebra, $\nabla \in \FdA$,
	\mbox{$\Delta_{1}, \Delta_{2} \in \IA$}. And let $\varphi 
	\in \mathrm{Fm}_{\langu_{\rsim}}$ be a formula such that
	\[
		\varphi = \psi(p_{1}, \dots, p_{n},
			q_{1} \vee \rsim q_{1}, \dots, q_{m} \vee \rsim q_{m} ),\tag{\#}
			\label{Rur}
	\]
	where
	$\psi(p_{1}, \dots, p_{n}, q_{1}, \dots, q_{m}) \in \Fmli$ and
	$\{ p_{1}, \dots, p_{n} \} \cap \{ q_{1}, \dots, q_{m} \} =
	\varnothing$.
	Then
	\[
		Tw(\bfA, \nabla, \Delta_{1}) \models \varphi
		\iff
		Tw(\bfA, \nabla, \Delta_{2}) \models \varphi
	\]
	
\end{lemma}

Note that in case $m = 0$ formulas of the form~\eqref{Rur} are just formulas of \Fmli.

\begin{proof}
	
	It is enough to prove that
	\[
		Tw(\bfA, \nabla, \Delta_{1}) \models \varphi
		\Longrightarrow
		Tw(\bfA, \nabla, \Delta_{2}) \models \varphi.
	\]
	
	Given a valuation
	$v \colon \mathrm{Prop} \rarr Tw(\bfA, \nabla, \Delta_{1})$,
	assume that
	$v(p_{i}) = (x_{i}, y_{i})$ and $v(q_{j}) = (a_{j}, b_{j})$.
	Then we have
	\begin{multline*}
		\pi_{1}(v(\varphi)) =\\
		=\psi(
			\pi_{1}(v(p_{1})), \dots, \pi_{1}(v(p_{n})),
			\pi_{1}(v(q_{1} \vee \rsim q_{1})),
			\dots,
			\pi_{1}(v(q_{m} \vee \rsim q_{m}))
		) = \\
		= \psi(
			x_{1}, \dots, x_{n},
			a_{1} \vee b_{1},
			\dots,
			a_{m} \vee b_{m}
		).
	\end{multline*}
	
	Let's show that
	$
		\psi(
			d_{1}, \dots, d_{n},
			c_{1}, \dots, c_{m}
		) = 1_{\bfA}
	$
	for all $d_{1}, \dots, d_{n} \in \bfA$ and
	$c_{1}, \dots, c_{m} \in \nabla$.
	Indeed, in view of $\pi_{1}(Tw(\bfA, \nabla, \Delta_{1})) = \bfA$,
	we can define the $Tw(\bfA, \nabla, \Delta_{1})$-valuation
	$v_{0}(p_{i}) := (d_{i}, d'_{i})$ for some $d'_{1}, \dots, d'_{n} \in
	\bfA$. The condition $c \in \nabla$ implies
	$c \vee \bot = c \in \nabla$ and
	$c \wedge \bot = \bot \in \Delta_{1}$, so we can define
	$ v_{0}(q_{j}) := ( c_{j} , \bot )$. Thus, we have by Remark~\ref{rem1}
	\[
		\pi_{1}(v_{0}(\varphi)) = \psi(
			d_{1}, \dots, d_{n},
			c_{1}, \dots, c_{m}
		) = 1_{\bfA}.
	\]
	
	Now for an arbitrary valuation
	$v \colon \mathrm{Prop} \rarr Tw(\bfA, \nabla, \Delta_{2})$ such that
	$v(p_{i}) = (e_{i}, f_{i})$ and $v(q_{j}) = (z_{j}, g_{j})$
	we have
	\[
		\pi_{1}(v(\varphi))
		=\psi(
			e_{1}, \dots, e_{n},
			z_{1} \vee g_{1},
			\dots,
			z_{m} \vee g_{m}
		) = 1_{\bfA},
	\] since $e_{1}, \dots, e_{n} \in \bfA$ and
	$z_{1} \vee g_{1}, \dots, z_{m} \vee g_{m} \in \nabla$. \qedhere
	
\end{proof}

\begin{theorem}\label{t415}
	
	Suppose that $L \in \mathcal{E}\Nf$ can be axiomatized as
	\[
		L = \Nf + X,
	\] where $X$ is some set of formulas of the form~\eqref{Rur}. Then
	$\taub L$ is a modal companion of $L$.
	
\end{theorem}

\begin{proof}
	
	The algebra $Tw(\TL, \nablaL, \DeltaL)$ is a subalgebra
	of the twist-structure $Tw(\TL, \nablaL, N(\DeltaL))$, so
	\[
		\mathrm{L} Tw(\TL, \nablaL, \DeltaL)
		\supseteq
		\mathrm{L} Tw(\TL, \nablaL, N(\DeltaL)).
	\]
	
	On the other hand, Lemma~\ref{l414} guarantees that for every
	$\varphi \in X$,
	\[
		Tw(\TL, \nablaL, \DeltaL) \models \varphi
		\Longrightarrow
		Tw(\TL, \nablaL, N(\DeltaL)) \models \varphi.
	\]
	
	Thus,
	\[
		L = \mathrm{L} Tw(\TL, \nablaL, \DeltaL)
		=
		\mathrm{L}
		Tw(\TL, \nablaL, N(\DeltaL)),
	\]
	so $Tw(\TL, \nablaL, N(\DeltaL))$ is
	a proper model of $L$ with closed ideal. Now, we apply
	Proposition~\ref{p411}. \qedhere
	
\end{proof}

\begin{remark}
	
	Theorem~\ref{t415} implies that the so-called {\em special extensions
	of} \Nf having the form $\Nf + L$
	and the {\em normal special extensions of} \Nf of the form
	$\Nf + L + \{ \neg \neg (p \vee \rsim p) \}$,
	where $L \in \EInt$, have modal companions.
	
\end{remark}

\subsection{Logics with no modal companions}

Consider the {\em Kleene \mbox{axiom}}\footnote{We call $\chi$ the Kleene axiom due to its similarity to the identity $x \wedge \rsim x \leql y \vee \rsim y$ that distinguishes Kleene algebras \cite{Cignoli} in the variety of De Morgan algebras. Note that the term "Kleene algebra" also has a different meaning. See, for example, \cite{Kozen2002}.} and the {\em modified Kleene axiom}:
\begin{align*}
\chi & = (p \wedge \rsim p) \rarr (q \vee \rsim q),\\
\chi' & = \neg \neg (p \wedge \rsim p) \rarr (q \vee \rsim q).
\end{align*}
The {\em Kleene logic} is defined as $\Nk := \Nf +
\{ \text{\raisebox{1pt}{$\chi$}}\}$.

It is not hard to see that
\[
	Tw(\bfA, \nabla, \Delta) \models \chi
	\quad \Longleftrightarrow \quad
	a \leql b \enskip \text{for all}\ a \in \Delta\ \text{and}\
	b \in \nabla.
\]
Thus, the twist-structure $Tw(\mathbf{3}, \nabla, \Delta)$ is a model of \Nk, where $\mathbf{3} = \langle \{ \bot, \heartsuit, 1 \}; \vee, \wedge,
\rarr, \bot \rangle$ with $\bot \leql \heartsuit \leql 1$,
$\nabla = \{ \heartsuit, 1 \}$ and $\Delta = \{ \bot, \heartsuit \}$.

Notice that $Tw(\mathbf{3}, \nabla, \Delta) \not\models \chi'$. Indeed, for the valuation $v$ such that $v(p) = (\heartsuit, 1)$, $v(q) = (\heartsuit, \bot)$ the following holds:
\[
	\pi_1(v(\chi')) = \neg \neg \heartsuit \rarr \heartsuit = 1 \rarr
	\heartsuit = \heartsuit \not= 1.
\] So $\chi' \not\in \Nk$.

\begin{lemma}\label{l421}
	
	If $M \in \EBsf$ and $\Tb \chi \in M$, then $\Tb \chi' \in M$.
	
\end{lemma}

\begin{proof}
	
	We have $\WM \models \Tb \chi$, where \WM is
	the Lindenbaum twist-structure of $M$. Since
	$
		\Tb \chi =
		\Box(
			(\Box p \wedge \Box \rsim p)
			\rarr
			(\Box q \vee \Box \rsim q)
		)
	$, we also have $\WM \models
	(\Box p \wedge \Box \rsim p) \rarr (\Box q \vee \Box \rsim q)$,
	which means that \mbox{$\Box x \leql (\Box a \vee \Box b)$} for all
	$x \in \DeltaM$ and $(a, b) \in \WM$.
	
	Since $\DeltaM \in \mathcal{I}_{\Diamond}(\TM)$, if $x \in \DeltaM$,
	then $\Diamond \Box x \in \DeltaM$, and we have that
	$\Box \Diamond \Box x \leql (\Box a \vee \Box b)$ for all
	$x \in \DeltaM$ and $(a, b) \in \WM$. Thus,
	\[
		\WM \models
		\Box(
			\Box \neg \Box \neg (\Box p \wedge \Box \rsim p)
			\rarr
			(\Box q \vee \Box \rsim q)
		),
	\]
	i.e. $\Tb \chi' \in M$. \qedhere
	
\end{proof}

We have that $\chi' \not\in \Nk$, but $\Tb \chi \in \taub \Nk$, which implies
$\Tb \chi' \in \taub \Nk$ by Lemma~\ref{l421}. We obtain that
$\taub \Nk$ is not a modal companion of \Nk, so \Nk has no modal companions.

Moreover, $\mathbf{3}$ is a model of each $L$ from the class $\{ L \in \mathcal{E}\Int \mid
\Int \subseteq L \subseteq \mathsf{HT} =
\Int + \{ p \vee (p \rarr q) \vee \neg q \} \} = \EInt
\setminus \{ \mathsf{CL}, \Fmli \}$ whose cardinality is a continuum. So
$\chi' \not\in \Nk + L$, but $\Tb \chi' \in \taub(\Nk + L)$.
And we have a continuum of logics that do not have modal companions.

In conclusion we want to raise some natural questions to explore them later:
\begin{enumerate}[1.]
	
	\item If $L \in \ENf$ has a modal companion, does it have the greatest modal companion? If yes, how to axiomatise the greatest modal companion modulo $\taub L$?
	
	\item Is there a criteria when $L$ has modal companions or not?
	
	\item Whether the set of all \Nf-extensions, which have modal companions, forms a sublattice in $\ENf$?
	
\end{enumerate}

\paragraph*{{\bfseries Acknowledgements.}} This work was supported by the Russian Science Foundation under grant no.~23-11-00104,
\href{https://rscf.ru/en/project/23-11-00104/}{https://rscf.ru/en/pro\-ject/23-11-00104/} .
Taking this opportunity, the author would also like to thank his supervisor, Sergei P.\ Odintsov, for the problem formulation and the attention to this work.

\end{document}